\documentclass[12pt]{amsart}
\usepackage{hyperref}
\usepackage{amsfonts}
\usepackage{bbm}
\usepackage{esint}
\usepackage{mathrsfs}
\usepackage{amssymb,amsmath,amsfonts,amsthm}
\usepackage{color}
\usepackage{xcolor}
\numberwithin{equation}{section}

\DeclareMathOperator{\supp}{supp}

\allowdisplaybreaks

\theoremstyle{plain}
\newtheorem{theorem}{Theorem}[section]

\newcommand{\norm}[1]{\left\Vert#1\right\Vert}

\theoremstyle{definition}
\newtheorem{definition}[theorem]{Definition}

\allowdisplaybreaks

\newcommand{\e}{\varepsilon}

\newtheorem{thm}{Theorem}[section]

\begin{document}

\setcounter{tocdepth}{3}
\allowdisplaybreaks

\title[Commutator of Riesz transforms in the Dunkl setting]
{Riesz transforms and commutators\\ in the Dunkl setting}

\author[Y. Han, M.-Y. Lee, J. Li, and B.D Wick]{Yongsheng Han, Ming-Yi Lee, Ji Li and Brett D. Wick}

\thanks{The second author is supported by MOST 110-2115-M-008-009-MY2. 
The third author is supported by the
Australian Research Council under Grant No.~ARC DP220100285 and NNSF 12171221.  B.D. Wick's research partially supported in part by NSF grant NSF-DMS-1800057 as well as ARC DP190100970.}



\begin{abstract}
In this paper we characterise the optimal pointwise size and regularity estimates for the Dunkl Riesz transform kernel involving both the Euclidean metric and the Dunkl metric, where these two metrics are not equivalent. We further establish a suitable version of the pointwise kernel lower bound of the Dunkl Riesz transform via the Euclidean metric only. Then we show that the lower bound of commutator of the Dunkl Riesz transform is with respect to the BMO space associated with the Euclidean metric, and that the upper bound is respect to  the BMO space associated with the Dunkl metric. Moreover, the compactness and the two types of VMO are also addressed. 
\end{abstract}
\maketitle
\section{Introduction}

The classical Fourier transform, initially defined on $L^1(\mathbb{R}^{N}),$ extends to an isometry of $L^2(\mathbb{R}^{N})$ and commutes with translation, dilation and rotation groups. To study the differential operators associated to reflection groups, Dunkl in \cite{Du1, Du2} introduced a similar transform, the Dunkl transform, which enjoys properties similar to the classical Fourier transform.  The Dunkl transform is given by 
$${ \mathcal{F}_\kappa} f(\xi):=c_\kappa^{-1}\int_{\mathbb{R}^{N}} E(-i\xi, x)f(x){d\omega(x)},$$
where the usual character $e^{-i\langle x,y\rangle}$ is replaced by 
$E(x, y):=\int_{\mathbb R^N} e^{\langle \eta, y\rangle}d\mu_x(\eta)$. Here
$\mu_x$ is a probability measure supported in the convex hull ${\mathcal O}(x)$ of the $G$-orbit of $x$ and the measure $\omega$ are invariant under a finite reflection group $G$ on $\mathbb{R}^{N}$ and 
$c_\kappa=\int_{\mathbb R^N} e^{-\frac{\|x\|^2}2}d\omega(x).$
Corresponding to the Dunkl transform,
 the
Dunkl translation operator $\tau_x$ is defined on  $L^2(\mathbb R^N, d\omega)$  by,
\begin{eqnarray}\label{dutr}
\mathcal{F}_\kappa (\tau_x(f))(y)=
E(ix,y)\mathcal{F}_\kappa f(y), \quad y\in\mathbb{R}^N.
\end{eqnarray}
See also \cite{BCV, deJ, R1, R2, R3, TX1} for more topics related
to the Dunkl setting.

Parallel to classical singular integrals, there is a natural Riesz transform in this Dunkl setting. The case $N=1$, goes back to the work of S. Thangavelu and Y. Xu \cite{TX}, where they established the $L^p$-boundedness of the associated Riesz transform in the Dunkl setting.  This was extended to the case of general dimension $N$ by Amri and Sifi \cite{AS}. See also \cite{DH2,DH4} for singular integrals and multipliers. 

Here we recall the setting of $\mathbb R^N$. 
Consider the Euclidean space $\mathbb{R}^{N}$   equipped with the standard inner product
$\langle x,y\rangle=\sum_{j=1}^Nx_jy_j$
and the corresponding Euclidean norm $\|x\|=\left\{\sum_{j=1}^N|x_j|^2\right\}^\frac{1}{2}.$ Let $B(x,r):=\{y\in \mathbb{R}^{N}:\|x-y\|<r\}$ be the Euclidean ball with center  $ x\in\mathbb{R}^{N}$  and radius   $r>0$.

In $\mathbb R^N$, the reflection $\sigma_\alpha$ with respect to the hyperplane $\alpha_\bot$ orthogonal to a nonzero vector $\alpha$ is given by
$$ \sigma_\alpha(x) = x- 2 { \langle x,\alpha\rangle\over \|\alpha\|^2} \alpha.$$
A finite set $R\subset \mathbb R^N\backslash\{0\}$ is called a \textit{root system} if $\sigma_\alpha(R) = R$ for every $\alpha\in R$.
Let $R$ be a root system in $\mathbb{R}^{N}$ normalized so that $\langle\alpha,\alpha\rangle=2$ {for $\alpha\in R$ and $G$ the finite reflection group generated by the reflections $\sigma_\alpha$ ($\alpha\in R$), where $\sigma_{\alpha}(x)=x-\langle\alpha,x\rangle\alpha$ for $x\in\mathbb{R}^{N}$.}  Corresponding to this reflection group, we denote by $\mathcal{O}(x)$ 
the $G$-orbit of a point
$ x\in\mathbb{R}^{N}$.  There is a natural metric between two $G$-orbits $\mathcal{O}(x)$ and $\mathcal{O}(y)$, given by
$$d(x,y):=\min\limits_{\sigma\in G}\|x-\sigma(y)\|.$$
It is clear that $d(x,y)\leq \|x-y\|$ and it is possible that for certain $x,y\in\mathbb R^N$, $d(x,y)=0$ while $\|x-y\|>0$.

For a \textsl{multiplicity function} $\kappa$ defined on {$R$} (invariant under $G$), let
\begin{eqnarray}\label{measure}
d\omega(x)=\prod_{\alpha\in R}|\langle\alpha,x\rangle|^{\kappa(\alpha)}dx
\end{eqnarray}
 be the associated measure in $\mathbb{R}^{N}$, where, here and subsequently, $dx$ stands for the Lebesgue
measure in $\mathbb{R}^{N}$.

The Dunkl Riesz transforms $R_j$, $j=1,2,\ldots,N$, 
are defined on $L^2(\mathbb R^N, d\omega)$ by
\begin{align}\label{Def Dunkl Riesz 1}
 R_j(f)(x) = d_\kappa \lim_{\epsilon\to0} \int_{|y|>\epsilon} \tau_x(f)(-y) { y_j \over \|y\|^{p_\kappa}} d\omega(y),\quad x\in\mathbb R^N, 
\end{align} 
where 
$d_\kappa = 2^{p_k-1\over2} {\Gamma( {p_\kappa\over2} ) \over\sqrt{\pi}}$, {$p_\kappa = \gamma_\kappa+N+1 $ and $\gamma_\kappa = \sum_{\alpha\in R} \kappa(\alpha)$}.
In \cite{AS} the authors obtained an explicit expression for the kernel ${R_j}(x,y)$ through which \eqref{Def Dunkl Riesz 1} can be represented as 
$$ R_j(f)(x) =\int_{\mathbb R^N} {R_j}(x,y) f(y) d\omega(y).$$  

Indeed, For $x,y\in \mathbb R^N$ and $\eta$ in the convex hull ${\mathcal O}(x)$, set
$A(x,y,\eta)=\sqrt{\|x\|^2+\|y\|^2-2\langle y,\eta\rangle}.$
Denote by
$$K^{(1)}_j(x,y)=\int_{\mathbb R^N} \frac{\eta_j-y_j}{A^{p_\kappa}(x,y,\eta)}d\mu_x(\eta)$$
and 
$$K_j^{(\alpha)}(x,y)=\frac1{\langle y,\alpha\rangle}\int_{\mathbb R^N} 
   \bigg[\frac1{A^{p_\kappa-2}(x,y,\eta)}-\frac1{A^{p_\kappa-2}(x,\sigma_\alpha\cdot y,\eta)}\bigg]d\mu_x(\eta),\quad \alpha\in R_+.$$	
The kernel $R_j(x,y)$ is given by
$$R_j(x,y):=d_\kappa\left\{K^{(1)}_j(x,y) + \sum_{\alpha\in R_+} \frac{\kappa(\alpha)\alpha_j}{p_\kappa-2}K_j^{(\alpha)}(x,y)\right\}.$$   
Moreover, $R_j(x,y)$ satisfies the H\"ormander condition: there exists $C>0$ such that 
$$ \int_{ d(x,y)\geq 2\|y-y_0\| } | {R_j}(x,y) -{R_j}(x,y_0) | d\omega(x)\leq C,\qquad y,y_0\in\mathbb R^N. $$

However, the H\"ormander condition alone is insufficient to bring in recent progress and techniques in harmonic analysis to this Dunkl Riesz transform, such as the sparse domination and sharp quantitative weighted estimate \cite{H,HRT,La,Lerner2}, and the boundedness and compactness of commutators (and its two weight setting) \cite{HLW,LL}.  And since we are after more refined estimates other than just the $L^p$ boundedness, we rectify this lack of information, through the first main result in this paper, the optimal  pointwise size and smoothness estimate of the Riesz transform kernel.
\begin{theorem}\label{smooth r}
There exists a constant $C$ such that for $j=1,2,\ldots,N$ and for every $x,y$ with $d(x,y)\not=0$,
\begin{equation}\label{size}
|R_j(x,y)|\le C\frac{d(x,y)}{\|x-y\|}\frac1{\omega(B(x,d(x,y)))},\hskip7cm
\end{equation}
\begin{equation}\label{smooth y}
|R_j(x,y)-R_j(x,y')|\le C\frac{\|y-y'\|}{\|x-y\|}\frac1{\omega(B(x,d(x,y)))} \qquad {\rm for}\ \|y-y'\|\le  d(x,y)/2,
\end{equation}
\begin{equation}\label{smooth x}
|R_j(x',y)-R_j(x,y)|\le C\frac{\|x-x'\|}{\|x-y\|}\frac1{\omega(B(x,d(x,y)))}\qquad {\rm for}\ \|x-x'\|\le d(x,y)/2.
\end{equation}
\end{theorem}

With this result, the door opens to many other questions about the Dunkl Riesz transforms.  It would seem that all properties of Dunkl Riesz transform would become clear since the above pointwise size and smoothness estimates are in the standard form of Calder\'on--Zygmund operators. However, a problem still exists in that there are two different, though related, metrics appearing in the estimates (the same comment holds true for the H\"ormander condition) and these metrics are not equivalent.  Even with these more standard Calder\'on-Zygmund estimates, the Dunkl Riesz transforms does not fall into the classical frame of Calder\'on--Zygmund theory.

A natural question arises: ``{\it What is the right version of the corresponding BMO space in the Dunkl setting?}''   In \cite{Dz}, Dziuba\'nski characterised the {Dunkl} Hardy space (in terms of the Euclidean metric and Dunkl measure $d\omega$)  via the Dunkl Riesz transforms (see also \cite{ADH}). We now investigate the BMO space in this Dunkl setting. A typical question is to consider the BMO space and the commutator of the Dunkl Riesz transform $[b,R_j]$. As is well-known, in the classical setting, Coifman, Rochberg and Weiss \cite{CRW} characterised the boundedness of commutators of Riesz transform via the space of BMO functions. However, due to the conflict of metrics of the size and regularity of the kernel as in Theorem \ref{smooth r}, the approach in \cite{CRW} and the modern methods as in \cite{H2, LOR} do not directly apply.

The second main result of this paper is to establish the link between boundedness of the commutator of the Dunkl Riesz transform $[b,R_j]$ and a corresponding BMO space, and show that the BMO defined via the Euclidean metric ball and the associated measure $d\omega(x)$ is the lower bound
of $[b,R_j]$ and the one with $d(x,y)$ is the upper bound of $[b,R_j]$ in the Dunkl setting. 
Before addressing this, we first investigate the pointwise kernel lower bound for the Dunkl Riesz transform as follows:
\begin{theorem}\label{th pointwise lower}
For $j=1,2,\ldots,N$ and for every ball $B=B(x_0,r)\subset \mathbb R^N$, there is another ball $\widetilde B=B(y_0,r)$ such that $\|x_0-y_0\|=5r$, and that for every  $(x,y)\in B\times \widetilde B$, 
$$|R_j(x,y)|\ge \frac C{\omega(B(x_
     0,r))}.$$
\end{theorem}

To state our result on commutator, we recall the BMO space in the Dunkl setting as
$$ {\rm BMO}_{Dunkl}(\mathbb R^N) =\{ b\in L^1_{loc}(\mathbb R^N, d\omega): \|b\|_*<\infty \},$$
where 
$$ \|b\|_* = \sup_{B \subset\mathbb R^N} {1\over \omega(B)}\int_B | b(x)-b_B |d\omega(x)<\infty $$
with the supremum is taken over all Euclidean balls $B=B(y,r) =\{ z\in \mathbb R^N: \|z-y\|<r\}$ and 
\begin{equation}\label{mean ball}
b_B  = {1\over \omega(B)}\int_B b(x)d\omega(x).  
\end{equation}
We also recall the ${\rm BMO}_d(\mathbb R^N)$ space associated with $d(x,y)$ as
$${\rm BMO}_d(\mathbb R^N) =\{ b\in L^1_{loc}(\mathbb R^N, d\omega): \|b\|_d<\infty \},$$
where 
$$ \|b\|_d = \sup_{B \in\mathbb R^N} {1\over \omega(\mathcal{O}(B))}\int_{\mathcal{O}(B)} | b(x)-b_{\mathcal{O}(B)} |d\omega(x)<\infty. $$
Note that ${\rm BMO}_d(\mathbb R^N)\subsetneq  {\rm BMO}_{Dunkl}(\mathbb R^N)$ (see for example \cite{JL}).
We have the first main result.
\begin{theorem}\label{commutator}
Suppose  $b\in L^1_{loc}(\mathbb R^N, d\omega)$. Consider the commutator of the Dunkl Riesz transform $[b,R_j]$, defined by
$ [b,R_j](f)(x) = b(x)R_j(f)(x) - R_j(bf)(x). $
Suppose $b\in {\rm BMO}_d$. Then for $1<p<\infty$, $[b,R_j]$ is bounded on $L^p(\mathbb R^N,d\omega)$ with 
$$ \|[b,R_j]\|_{ L^p(\mathbb R^N,d\omega)\to L^p(\mathbb R^N,d\omega) } \lesssim \|b\|_d.$$ 
Conversely, if $[b,R_j]$ is bounded on $L^p(\mathbb R^N,d\omega)$ for some $1<p<\infty$, then $b\in  {\rm BMO}_{Dunkl}(\mathbb R^N)$ with 
$$ \|b\|_* \lesssim  \|[b,R_j]\|_{ L^p(\mathbb R^N,d\omega)\to L^p(\mathbb R^N,d\omega) } . $$
\end{theorem}

%\begin{remark}
%Since the Dunkl Riesz transform kernel has the pointwise lower bound via the Euclidean metric only,
%it is natural to obtain that $ \|b\|_* \lesssim  \|[b,R_j]\|_{ L^p(\mathbb R^N,d\omega)\to L^p(\mathbb R^N,d\omega) } . $ For the upper bound of $\|[b,R_j]\|_{ L^p(\mathbb R^N,d\omega)\to L^p(\mathbb R^N,d\omega) }$, we give an example that $b\in  {\rm BMO}_{Dunkl}(\mathbb R^N)$ but
% $\|[b,R_j]\|_{ L^p(\mathbb R^N,d\omega)\to L^p(\mathbb R^N,d\omega) }=\infty$ for some $p$ close to $1$ (see the last section). Hence, one can not expect the upper bound in the above theorem to hold with the estimate $\|b\|_*$.  
 
%Because of the lower bounded information we have on the Dunkl Riesz kernels in terms of the Euclidean metric, we are unable to address the question as to if the potential lower bound: $\|b\|_d\lesssim \|[b,R_j]\|_{ L^p(\mathbb R^N,d\omega)\to L^p(\mathbb R^N,d\omega) }$ holds true.   This would be a natural question to explore.
%\end{remark}

With the boundedness of the commutator now completely understood we can additionally consider additional operator theoretic conditions of the commutator.  In particular, we obtain information about the compactness of these commutators.  To do so, we define the VMO space in the Dunkl setting as follows:
$$ {\rm VMO}_{Dunkl}(\mathbb R^N) =\{ b\in {\rm BMO}_{Dunkl}(\mathbb R^N): (1)-(3) \ holds\} $$
where 
$$ (1)\  \lim_{r\to 0}\sup_{B \subset\mathbb R^N, r_B=r} {1\over \omega(B)}\int_{B} | b(x)-b_{B} |d\omega(x)=0, $$
$$ (2)\  \lim_{r\to \infty}\sup_{B \subset\mathbb R^N, r_B=r} {1\over \omega(B)}\int_{B} | b(x)-b_{B} |d\omega(x)=0, $$
$$ (3)\  \lim_{r\to \infty}\sup_{B \subset\mathbb R^N, B\cap B(0,r)=\emptyset } {1\over \omega(B)}\int_B | b(x)-b_B |d\omega(x)=0.$$
We  define the VMO space associated the Dunkl metric as follows:
$$ {\rm VMO}_d(\mathbb R^N) =\{ b\in {\rm BMO}_d(\mathbb R^N): (4)-(6) \ holds\} $$
where 
$$ (4)\  \lim_{r_B\to 0}\sup_{{\mathcal O}(B) \subset\mathbb R^N} {1\over \omega({\mathcal O}(B))}\int_{{\mathcal O}(B)} | b(x)-b_{{\mathcal O}(B)} |d\omega(x)=0, $$
$$ (5)\  \lim_{r_B\to \infty}\sup_{{\mathcal O}(B) \subset\mathbb R^N} {1\over \omega({\mathcal O}(B))}\int_{{\mathcal O}(B)} | b(x)-b_{{\mathcal O}(B)} |d\omega(x)=0, $$
$$ (6)\  \lim_{r\to \infty}\sup_{B \subset\mathbb R^N, {\mathcal O}(B)\cap B(0,r)=\emptyset } {1\over \omega({\mathcal O}(B))}\int_{{\mathcal O}(B)} | b(x)-b_{{\mathcal O}(B)} |d\omega(x)=0.$$

The result we obtain is the following characterisation of compactness and equivalence with VMO, which can be seen as an extension of the work by Uchiyama for the standard Calder\'on--Zygmund operators in the Euclidean setting, see \cite{Uch78}.

\begin{theorem}\label{compact}
Suppose $b\in {\rm BMO}_{Dunkl}(\mathbb R^N)$. If $b\in {\rm VMO}_{d}(\mathbb R^N)$, then for $1<p<\infty$, $[b,R_j]$ is compact on $L^p(\mathbb R^N,d\omega)$.
Conversely, if $[b,R_j]$ is compact on $L^p(\mathbb R^N,d\omega)$, then $b\in {\rm VMO}_{Dunkl}(\mathbb R^N)$.
\end{theorem}

Based on the properties for the Dunkl transform, the Dunkl Poisson semigroup and the Dunkl Riesz transform, we obtain the upper bound via the Dunkl Poisson extension and Carleson measure estimates. The proof strategy we utilize was used when studying the standard Laplacian and classical Riesz transforms via the Poisson extension in \cite{LS}. The lower bound follows from proving that the Dunkl Riesz transform kernel satisfies the non-degenerate condition (see for example the standard setting \cite{H2,LOR}) via the Euclidean metric.

The paper is organised as follows. We will prove Theorems \ref{smooth r} and  \ref{th pointwise lower} in Section 2. The upper bound for $[b,R_j]$ will be given in Section 3, and then the lower bound for $[b,R_j]$ in Section 4.  Compactness is dealt with in Section 5.  %Finally in Section 6, an example is provided to show that the result obtained in natural in terms of the boundedness of the commutator.

\section{Proof of Theorems \ref{smooth r} and \ref{th pointwise lower}}

Consider the Euclidean space $\mathbb{R}^{N}$   equipped with the standard inner product
and the corresponding norm.  
Let $B(x,r):=\{y\in \mathbb{R}^{N}:\|x-y\|<r\}$ stand for the ball with center  $ x\in\mathbb{R}^{N}$  and radius   $r>0$.
In $\mathbb R^N$, the reflection $\sigma_\alpha$ with respect to the hyperplane $\alpha_\bot$ orthogonal to a nonzero vector $\alpha$ is given by
$$ \sigma_\alpha(x) = x- 2 { \langle x,\alpha\rangle\over \|\alpha\|^2} \alpha.$$
A finite set $R\subset \mathbb R^N\backslash\{0\}$ is called a root system if $\sigma_\alpha(R) = R$.

Let $R$ be a root system in $\mathbb{R}^{N}$ normalized so that $\langle\alpha,\alpha\rangle=2$ for $\alpha\in R$ and with $R_+$ a fixed positive subsystem, and $G$ the finite reflection group generated by the reflections $\sigma_\alpha$ ($\alpha\in R$). We shall denote by $\mathcal{O}(x)$, resp. $\mathcal{O}(B)$ the G-orbit of a point
$ x\in\mathbb{R}^{N}$, resp. a subset $B\subset \mathbb{R}^{N}$.

We denote by $\textbf{N} = N +\sum\limits_{\alpha\in R}\kappa(\alpha)$
 the homogeneous dimension of the
system. The measure $d\omega$ as in \eqref{measure} satisfies that 
$$\omega(B(tx, tr)) = t^\textbf{N}\omega(B(x, r))$$
and that 
there is a constant $C > 0$ such that
\begin{eqnarray}\label{doubling condition}\omega(B(x, 2r)) \leq C\omega(B(x, r))< \infty\end{eqnarray}
for all $x\in \mathbb{R}^{N}$, $t,r>0$.
Moreover,
\begin{equation}\label{w-double 2}
C^{-1}\bigg(\frac{r_2}{r_1}\bigg)^N\le \frac{w(B(x,r_2))}{w(B(x,r_1))}\le C\bigg(\frac{r_2}{r_1}\bigg)^\textbf{N}
\qquad\text{for}\quad 0<r_1<r_2.
\end{equation}
 and
$$\int_{\mathbb{R}^{N}}f(x)d\omega(x)=\int_{\mathbb{R}^{N}}\frac{1}{t^\textbf{N}}f\Big(\frac{x}{t}\Big)d\omega(x)$$
for $f \in L^1(\mathbb{R}^N, d\omega(x))$, $t>0$. By \eqref{doubling condition}, it is easy
to see $\omega(B(x, \|x-y\|)) \ {\approx}\ \omega(B(y,  \|x-y\|))$.

Recall that $$d(x,y):=\min\limits_{\sigma\in G}\|x-\sigma(y)\|$$
denotes the distance between two G-orbits $\mathcal{O}(x)$ and $\mathcal{O}(y)$. Obviously, 
$$\mathcal{O}(B(x,r))=\bigcup_{\sigma\in G}B(\sigma(x),r)=\{y\in \mathbb{R}^{N}:d(x,y)<r\}$$
and $$\omega(B(x,r))\leq \omega\big(\mathcal{O}(B(x,r))\big)\leq |G|\omega(B(x,r)).$$
See in \cite{ADH, DH2}.

The Dunkl operators $T_\xi$, introduced in \cite{Du1}, is the following $\kappa$-deformation of the directional derivative $\partial_\xi$ by a difference operator:
$$T_\xi f(x)=\partial_\xi f(x)+\sum_{\alpha\in R}\frac {\kappa(\alpha)}2 \langle \alpha, \xi\rangle \frac{f(x)-f(\sigma_\alpha(x))}{\langle \alpha,x\rangle}.$$

For fixed $y\in \mathbb R^N$ the Dunkl kernel $E(x,y)$ is a unique solution of the system
$$T_\xi f=\langle \xi, y\rangle f,\quad f(0)=1.$$
Let $e_j, j=1,\ldots,N$ denote the canonical orthonormal basis in $\mathbb R^N$ 
and let $T_j=T_{e_j}$.  
In particular $$T_{j,x}E(x,y)=y_jE(x,y),$$
where $T_{j,x}$ denotes the action of $T_j$ with respect to the variable $x$.

For $f\in L^1(\mathbb R^N, d\omega)$ (the Lebesgue space with respect to the measure $\omega$)
the Dunkl transform  is defined by
$$\mathcal{F}_k(f)(\xi)=\frac{1}{c_k}
\int_{\mathbb{R}^N}\,E_k(-i\,\xi,x)\!f(x)d\omega(x),\quad c_k\,=\int_{\mathbb{R}^N}\!e^{-\frac{\|x\|^2}2}\,d\omega(x).
$$

The Dunkl translation $\tau_xf$ of a function $f\in {\mathcal S}(\mathbb R^N)$ by $x\in \mathbb R^N$ is defined by
$$\tau_xf(y)=c_k^{-1}\int_{\mathbb R^N} E(i\xi,x)E(i\xi,y)\mathcal{F}_kf(\xi)d\omega(\xi).$$
If $f$ is a  continuous radial function in $L^2(\mathbb{R}^N, \omega)$ with
$f(y)=\widetilde{f}(\norm{y})$, then
$$
              \tau_x(f)(y)=\int_{\mathbb{R}^{N}}\widetilde{f}\bigg(\;\sqrt{\norm{x}^2+\norm{y}^2+2\langle y,\eta\rangle}\;\bigg)d\mu_x(\eta).
$$
This formula is first proved by   M. R\"{o}sler \cite{R2}  for  $f\in \mathcal{S}(\mathbb{R}^N)$  and   recently is
extended to {radial} continuous functions by F. Dai and H. Wang \cite {DW}.
 \par We collect below some useful facts:
  \begin{itemize}
   \item
  [(i)] For all $x,y\in \mathbb{R}^N$,
$$\tau_x(f)(y)=\tau_y(f)(x).$$
   \item  [(ii)]For all $x,\xi\in \mathbb{R}^N$ and $f\in\mathcal{S}(\mathbb{R}^N)$,
   $$T_\xi \tau_x(f)=\tau_x T_\xi(f).$$
   \item [(iii)]
   For all $x\in \mathbb{R}^N$ and $f,\;g\in L^2(\omega)$,
$$\int_{\mathbb{R}^N}\tau_x(f)(-y)g(y)d\omega(y)=\int_{\mathbb{R}^N}f (y) \tau_xg(-y) d\omega(y).
$$

           \item [(iv)] For all $x\in \mathbb{R}^N$ and  $1\leq p\leq2$, the operator $\tau_x$ can be extended to all radial
           functions $f$ in $L^p(\mathbb{R}^N,\omega)$  and the following holds
$$ \|\tau_x(f)\|_{L^p(\omega)}\leq \|f\|_{L^p(\omega)}.$$
  \end{itemize}

The Dunkl Laplacian associated with $G$ and $\kappa$ is the differential-difference operator $\Delta=\sum_{j=1}^N T_j^2$, which acts on $C^2(\mathbb R^N)$-functions by 
$$\Delta f(x)= \vartriangle_{\mbox{eucl}} f(x)+ \sum_{\alpha\in R} \kappa(\alpha)\delta_\alpha f(x),$$
$$\delta_\alpha f(x)=\frac{\partial_\alpha f(x)}{\langle\alpha, x\rangle}-\frac{\|\alpha\|^2}2 \frac{f(x)-f(\sigma x)}{\langle\alpha, x\rangle^2}.$$
The operator $\Delta$ is essentially self-adjoint on $L^2(\mathbb{R}^N, \omega)$. The semigroup has the form
$$H_t(f)(x) = e^{t\Delta} f(x)=\int_{\mathbb R^N} h_t(x,y)f(y)d\omega(y),$$
where the heat kernel
$$h_t(x,y)=\tau_xh_t(-y),\qquad {\rm with}\  h_t(x)=c_\kappa^{-1} (2t)^{-\textbf{N}/2}e^{-\|x\|^2/(4t)},$$
is a $C^\infty$-function of all variables $x,y\in \mathbb R^N, t>0,$ and satisfies
$$0<h_t(x,y)=h_t(y,x), \qquad \int_{\mathbb R^N} h_t(x,y)d\omega(y)=1.$$

Set $$V(x,y,r):=\max\{\omega(B(x,r)),\omega(B(y,r))\}.$$
The following theorem was proved in \cite[Theorem 4.1]{ADH}.
\begin{theorem}[\cite{ADH}]\label{heat kernel}\,

  \begin{itemize}
  \item[(a)] There are constants $C,c>0$ such that
  $${1\over C{\min\{\omega(B(x,\sqrt t)), \omega(B(y,\sqrt t))\}}}e^{-c\|x-y\|^2/t}\leq |h_t(x,y)|\le CV(x,y,\sqrt t)^{-1}e^{-cd(x,y)^2/t},$$
  for every $t>0$ and for every $x,y\in \mathbb R^N$.
  \item[(b)] There are constants $C,c>0$ such that
  $$|h_t(x,y)-h(x,y')|\le C\bigg(\frac{\|y-y'\|}{\sqrt t}\bigg)V(x,y,\sqrt t)^{-1}e^{-cd(x,y)^2/t},$$
  for every $t>0$ and for every $x,y,y'\in \mathbb R^N$ such that $\|y-y'\|<\sqrt t$. 
   \end{itemize}
\end{theorem}

We now recall the Riesz transforms in the Dunkl setting defined by 
$$\mathcal F_\kappa({R_j f})(\xi)=-i\frac{\xi_j}{\|\xi\|}\mathcal F_\kappa({f})(\xi)\qquad\mbox{for } j=1,2,\cdots,N.$$
Note that 
$$R_j f=-T_j(-\Delta)^{-1/2}f=-C_1\int_0^\infty T_je^{t\Delta} f \frac{dt}{\sqrt t},$$
where the integral converges in $L^2(\mathbb{R}^N,\omega)$ (See \cite[page 2391]{ADH}). 
{In \cite[Lemma 3.3]{DH2}}, for all $x,y\in \mathbb R^N$ and $t>0$,
$$T_jh_t(x,y)=\frac{y_j-x_j}{2t}h_t(x,y).$$
We write the Riesz transforms as follows:
$$R_j f(x)= \int_{\mathbb R^N} R_j(x,y)f(y)d\omega(y),$$
then the kernel $R_j(x,y)$ satisfies the following smoothness condition (1.4)-(1.6).

\begin{proof}[Proof of Theorem \ref{smooth r}]
To estimate the kernel $R_j(x,y),$ we recall the following estimates
for the Dunkl-heat kernel given in \cite[Theorem 3.1]{DH2}

\begin{itemize}
    \item[(a)] There are constants $C,c>0$ such that
    $$|h_t(x,y)|\leqslant C\frac1{V(x,y,\sqrt t)}\bigg(1+\frac{\|x-y\|}{\sqrt t}\bigg)^{-2}e^{-cd(x,y)^2/t},$$
    for every $t>0$ and for every $x,y\in \mathbb R^N$.
    \item[(b)] There are constants $C,c>0$ such that
    $$|h_t(x,y)-h(x,y')|\leqslant C\bigg(\frac{\|y-y'\|}{\sqrt t}\bigg)\frac1{V(x,y,\sqrt t)}\bigg(1+\frac{\|x-y\|}{\sqrt t}\bigg)^{-2}e^{-cd(x,y)^2/t},$$
    for every $t>0$ and for every $x,y,y'\in \mathbb R^N$ such that $\|y-y'\|<\sqrt t$.
\end{itemize}
We now estimate the kernel $R_j(x,y)$ as follows.

\begin{align*}
|R_j(x,y)|&\lesssim |y_j-x_j|\int_0^\infty \frac1{V(x,y,\sqrt t)}\frac t{\|x-y\|^2}e^{-cd(x,y)^2/t}\frac{dt}{t\sqrt t} \\
&\leq \frac1{\|x-y\|}\bigg(\int_0^{d(x,y)^2} + \int_{d(x,y)^2}^\infty\bigg) \frac1{V(x,y,\sqrt t)}e^{-cd(x,y)^2/t}\frac{dt}{\sqrt t}\\
&=: I_1 +I_2.
\end{align*}
For $t\leqslant d(x,y)^2$, by using the doubling condition we have
that
$$\omega(B(x,d(x,y)))\lesssim \Big(\frac {d(x,y)}{\sqrt
t}\Big)^{\mathbf N}\omega(B(x,\sqrt t))$$ and hence
$$
V(x,y,\sqrt t)^{-1}\lesssim \frac1{\omega(B(x,\sqrt t))}\lesssim
\Big(\frac {d(x,y)}{\sqrt t}\Big)^{\mathbf
N}\frac1{\omega(B(x,d(x,y)))}.
$$
We obtain
\begin{align*}
I_1&\lesssim   \frac1{\|x-y\|}\frac1{\omega(B(x,d(x,y)))} \int_0^{d(x,y)^2}\Big(\frac {d(x,y)}{\sqrt t}\Big)^{\mathbf N}e^{-cd(x,y)^2/t}\frac{dt}{\sqrt t}\\
&\lesssim   \frac1{\|x-y\|}\frac1{\omega(B(x,d(x,y)))} \int_0^{d(x,y)^2}\frac {d(x,y)^{\mathbf N}}{t^{\frac{1+\mathbf N}2}}\Big(\frac{t}{d(x,y)^2}\Big)^{\frac{1+\mathbf N}2}dt\\
&\lesssim  \frac{d(x,y)}{\|x-y\|}\frac1{\omega(B(x,d(x,y)))}.
\end{align*}
It is clear that for $t\geqslant d(x,y)^2$, by using the reversed
doubling condition, $$\Big(\frac {\sqrt
t}{d(x,y)}\Big)^N\omega(B(x,d(x,y)))\lesssim C\omega(B(x,\sqrt
t)),$$ we get
\begin{align*}
I_2&\lesssim \frac1{\|x-y\|}\int_{d(x,y)^2}^\infty \frac1{V(x,y,d(x,y))}\frac {d(x,y)^{N}}{t^{\frac{1+ N}2}}dt\\
&\lesssim  \frac{d(x,y)}{\|x-y\|}\frac1{\omega(B(x,d(x,y)))}.
\end{align*}
To see the smoothness estimates, we write

$$|R_j(x,y)-R_j(x,y')|
\le  C|y_j-y_j'| \int_0^\infty |h_t(x,y)|\frac{dt}{t\sqrt t}+|y_j'-x_j|\int_0^\infty |h_t(x,y)-h_t(x,y')|\frac{dt}{t\sqrt t}.$$
By the above method, 
$$|y_j-y_j'| \int_0^\infty |h_t(x,y)|\frac{dt}{t\sqrt t}\lesssim  \frac{\|y-y'\|}{\|x-y\|}\frac1{\omega(B(x,d(x,y)))}.$$
To obtain 
$$|R_j(x,y)-R_j(x,y')|\lesssim \frac{\|y-y'\|}{\|x-y\|}\frac1{\omega(B(x,d(x,y)))}
\quad \mbox{ for } \|y-y'\|<\frac 12d(x,y),$$
it suffices to show that
$$|y_j'-x_j|\int_0^\infty |h_t(x,y)-h_t(x,y')|\frac{dt}{t\sqrt t}\lesssim \frac{\|y-y'\|}{\|x-y\|}\frac1{\omega(B(x,d(x,y)))}\quad \mbox{ for } \|y-y'\|<\frac 12d(x,y).$$
If $\|y-y'\|<\frac 12d(x,y)$, then $\|y'-x\|\le \frac32\|x-y\|$. Hence
\begin{align*}
&|y_j'-x_j|\int_0^\infty |h_t(x,y)-h_t(x,y')|\frac{dt}{t\sqrt t}\\
&\qquad\le C\|x-y\|\int_0^\infty |h_t(x,y)-h_t(x,y')|\frac{dt}{t\sqrt t}\\
&\qquad\le C\|x-y\| \bigg(\int_0^{d(x,y)^2} +
\int_{d(x,y)^2}^\infty\bigg)
|h_t(x,y)-h_t(x,y')|\frac{dt}{t\sqrt t} \\
&\qquad=: I\!I_1 +I\!I_2.
\end{align*}
Note that if $\|y-y'\|<\sqrt t$, then the above condition (b) gives
 $$|h_t(x,y)-h(x,y')|\le C\bigg(\frac{\|y-y'\|}{\sqrt t}\bigg)\frac1{V(x,y,\sqrt t)}\bigg(1+\frac{\|x-y\|}{\sqrt t}\bigg)^{-2}e^{-cd(x,y)^2/t}.$$
If  $\|y-y'\|\ge \sqrt t$, then 
$$|h_t(x,y)-h(x,y')|\le \bigg(\frac{\|y-y'\|}{\sqrt t}\bigg)(|h_t(x,y)|+|h(x,y')|).$$
Since $\|y-y'\|<\frac 12d(x,y)$, we have $d(x,y)\approx d(x,y')$ and $\|x-y\|\approx \|x-y'\|$
and thus
\begin{align*}
I\!I_1&\lesssim  \|y-y'\|\|x-y\|\int_0^{d(x,y)^2}\frac1{V(x,y,\sqrt t)}\frac t{\|x-y\|^2}e^{-cd(x,y)^2/t}\frac{dt}{t^2}\\
&\lesssim \frac{\|y-y'\|}{\|x-y\|}\frac1{\omega(B(x,d(x,y)))} \int_0^{d(x,y)^2}\Big(\frac {d(x,y)}{\sqrt t}\Big)^N e^{-cd(x,y)^2/t}\frac{dt}{t}\\
&\lesssim \frac{\|y-y'\|}{\|x-y\|}\frac1{\omega(B(x,d(x,y)))} \int_0^{d(x,y)^2}\frac {d(x,y)^N}{t^{1+\frac N2}}\Big(\frac{t}{d(x,y)^2}\Big)^{1+\frac{N}2}dt\\
&\lesssim \frac{\|y-y'\|}{\|x-y\|}\frac1{\omega(B(x,d(x,y)))}.
\end{align*}
To estimate $I\!I_2$, we have $\|y-y'\|<\frac 12d(x,y)<\sqrt t$ and the above condition (b) gives
\begin{align*}
I\!I_2&\lesssim \frac{\|y-y'\|}{\|x-y\|}\int_{d(x,y)^2}^\infty
\frac1{V(x,y,\sqrt t)}
e^{-cd(x,y)^2/t}\frac{dt}{t}\\
&\lesssim \frac{\|y-y'\|}{\|x-y\|} \frac1{\omega(B(x,d(x,y)))}\int_{d(x,y)^2}^\infty \frac {d(x,y)^{N}}{t^{1+\frac N2}}dt\\
&\lesssim \frac{\|y-y'\|}{\|x-y\|}\frac1{\omega(B(x,d(x,y)))}.
\end{align*}

\noindent The estimate of the smoothness for $x$ variable is similar. The proof  of Theorem \ref{smooth r} is complete.
\end{proof}

We now prove the pointwise lower bounded of $R_j(x,y)$.

\begin{proof}[Proof of Theorem \ref{th pointwise lower}]
Let $B=B(x_0,r)$.
We choose $\widetilde B=B(y_0, r)$ with $\|x_0-y_0\|=5r$ and satisfy that $y_j-x_j\ge r$ and $\|x-y\|\approx r$ for $x\in B$  and $y\in \widetilde B$. 
Note that
$$R_j(x,y)=-C\int_0^\infty \frac{y_j-x_j}t h_t(x,y)\frac{dt}{\sqrt t}.$$
It is clear that
\begin{align*}
\int_0^\infty \frac1t h_t(x,y)\frac{dt}{\sqrt t}
&\gtrsim \int_0^\infty \frac1t\frac 1{\min\{\omega(B(x,\sqrt t)), \omega(B(y,\sqrt t))\}}e^{-c\|x-y\|^2/t}\frac{dt}{\sqrt t} \\
&= \bigg(\int_0^{\|x-y\|^2} + \int_{\|x-y\|^2}^\infty \bigg)\frac 1{\min\{\omega(B(x,\sqrt t)), \omega(B(y,\sqrt t))\}}e^{-c\|x-y\|^2/t}\frac{dt}{t\sqrt t} \\
&=:A_1+A_2.
\end{align*}
To estimate $A_1$, we use $s=\|x-y\|^2/t$ to get
\begin{align*}
A_1 &\ge \frac 1{\omega(B(x,\|x-y\|))}\int_0^{\|x-y\|^2}e^{-c\|x-y\|^2/t}\frac{dt}{t\sqrt t} \\
       &=\frac 1{\omega(B(x,\|x-y\|))}\|x-y\|^{-1}\int_1^\infty e^{-s}\frac{ds}{\sqrt s}.
\end{align*}
To estimate $A_2$, we use doubling condition to give
$$\omega(x,\sqrt t)\le \frac{t^{\textbf{N}/2}}{\|x-y\|^\textbf{N}}\omega(x,\|x-y\|)$$
and hence
\begin{align*}
A_2 &\ge \frac 1{\omega(B(x,\|x-y\|))}\int_{\|x-y\|^2}^\infty \frac{\|x-y\|^\textbf{N}}{t^{\textbf{N}/2}} e^{-c\|x-y\|^2/t}\frac{dt}{t\sqrt t} \\
       &=\frac 1{\omega(B(x,\|x-y\|))}\|x-y\|^{-1}\int_0^1 e^{-s}{s^{\frac{\textbf{N}+1}2}}ds.
\end{align*}
For $(x,y)\in B\times \widetilde B$, we obtain that
$$|R_j(x,y)|=C\bigg|\int_0^\infty \frac{y_j-x_j}t h_t(x,y)\frac{dt}{\sqrt t}\bigg|
     \gtrsim  \frac 1{\omega(B(x,\|x-y\|))}    \gtrsim  \frac 1{\omega(B(x,r))} \gtrsim  \frac 1{\omega(B(x_
     0,r))}.$$
The proof is completed.
\end{proof}

\section{Proof of Theorem \ref{commutator}: Upper bound of commutator}

The maximal function $Mf$ is defined as
$$Mf(x)=\sup_{x\in B}\frac1{\omega(B)}\int_B |f(y)|d\omega(y).$$
The sharp function $f^\sharp$ is defined as
$$f^\sharp(x)=\sup_{x\in B}\frac1{\omega(B)}\int_B |f(y)-f_B|d\omega(y),$$
where $f_B$ is defined in \eqref{mean ball}.

\begin{proof}[Proof of Theorem \ref{commutator}: upper bound of commutator]
Suppose $b\in BMO_d$, $1<p<\infty$ and $f$ in $L^p(\mathbb R^N,d\omega)$.

For any $x\in \mathbb R^N$ and for any ball $B=B(x_0,r)\subset \mathbb R^N$ containing $x$, we set 
$f = f_1+f_2$ with $f_1 = f\cdot \chi_{\mathcal{O}(5B)}$. 

Then for any $y\in B$, we have that
\begin{align*}
[b,R_j](f)(y) &= b(y) R_j(f)(y) - R_j(bf)(y)\\
&= (b(y)-b_{\mathcal{O}(B)}) R_j(f)(y) - R_j\big(( b-b_{\mathcal{O}(B)})f\big)(y)\\
&= (b(y)-b_{\mathcal{O}(B)}) R_j(f)(y) - R_j\big(( b-b_{\mathcal{O}(B)})f_1\big)(y)- R_j\big(( b-b_{\mathcal{O}(B)})f_2\big)(y)\\
&=: {\rm I}(y)+{\rm I\!I}(y)+{\rm I\!I\!I}(y).
\end{align*}

For  ${\rm I}(y)$ we have that 
\begin{align*}
&{1\over \omega(B)}\int_B |{\rm I}(y) - {\rm I}_B |d\omega(y) \\
&\leq {2\over \omega(B)}\int_B |{\rm I}(y)  |d\omega(y) \\
&={2\over \omega(B)}\int_B \Big|(b(y)-b_{\mathcal{O}(B)}) R_j(f)(y)  \Big |d\omega(y)\\
&\leq 2 \bigg({1\over \omega(B)}\int_B \Big|(b(y)-b_{\mathcal{O}(B)})   \Big |^{s'}d\omega(y)\bigg)^{1\over s'} \bigg({1\over \omega(B)}\int_B \Big|R_j(f)(y)  \Big |^sd\omega(y)\bigg)^{1\over s}\\
&\leq C \bigg({1\over \omega(\mathcal{O}(B))}\int_{\mathcal{O}(B)} \Big|(b(y)-b_{\mathcal{O}(B)})   \Big |^{s'}d\omega(y)\bigg)^{1\over s'} \bigg({1\over \omega(B)}\int_B \Big|R_j(f)(y)  \Big |^sd\omega(y)\bigg)^{1\over s}\\
&\leq C \|b\|_d \Big(M ( |R_jf|^s )(x)\Big)^{1\over s},
\end{align*}
where $s$ is chosen to satisfy $1<s<p<\infty$ and $s'$ is the conjugate of $s$.

For ${\rm I\!I}(y)$, since $R_j$ is bounded on $L^q(\mathbb R^N, d\omega), 1<q<\infty,$
we have
\begin{align*}
&{1\over \omega(B)}\int_B |{\rm II}(y) - {\rm II}_B |d\omega(y)\\
&\leq {2\over \omega(B)}\int_B |{\rm II}(y)  |d\omega(y) \\
&= {2\over \omega(B)}\int_B |R_j\big(( b-b_{\mathcal{O}(B)})f_1\big)(y) |d\omega(y) \\
&\lesssim \bigg({1\over \omega(B)}\int_B |R_j\big(( b-b_{\mathcal{O}(B)})f_1\big)(y) |^q d\omega(y) \bigg)^{\frac1q}\\
&\lesssim \bigg({1\over \omega(B)}\int_{\mathcal{O}(5B)} |b(y)-b_{\mathcal{O}(B)}|^q|f(y) |^q d\omega(y) \bigg)^{\frac1q}\\
&\lesssim \bigg({1\over \omega(B)}\int_{\mathcal{O}(5B)} |b(y)-b_{\mathcal{O}(B)}|^{qv'} d\omega(y) \bigg)^{\frac1{qv'}}
               \bigg({1\over \omega(B)}\int_{\mathcal{O}(5B)} |f(y)|^{qv} d\omega(y) \bigg)^{\frac1{qv}}\\
&\lesssim \|b\|_d\Big(M ( |f|^\beta )(x)\Big)^{1\over \beta},
\end{align*} 
where we have chosen $q,v\in (1,\infty)$ such that $1<qv<p<\infty$ and have set $\beta:=qv$.

Finally, we turn our attention to term ${\rm I\!I\!I}(y)$. For $w\in \mathbb R^N\setminus {\mathcal{O}(5B)}$, it is clear that for $y\in B$, $\|x_0-y\|\le \frac12d(x_0, w)$. Since $\omega(B(w,d(w,x_0)))\approx \omega(B(x_0,d(w,x_0)))$, we have
\begin{align*}
|{\rm I\!I\!I}(y)-{\rm I\!I\!I}(x_0)|
& = |R_j\big(( b-b_{\mathcal{O}(B)})f_2\big)(y)-R_j\big(( b-b_{\mathcal{O}(B)})f_2\big)(x_
0)| \\
&\le \int_{\mathbb R^N\setminus {\mathcal{O}(5B)}} |R_j(w,y)-R_j(w,x_0)||b(w)-b_
{\mathcal{O}(B)}||f(w)|d\omega(w)\\
&\lesssim \int_{\mathbb R^N\setminus {\mathcal{O}(5B)}} \frac{\|y-x_0\|}{\|w-x_0\|}\frac1{\omega(B(w,d(w,x_0)))}|b(w)-b_{\mathcal{O}(B)}
||f(w)|d\omega(w)\\
&\lesssim r\bigg(\int_{\mathbb R^N\setminus {\mathcal{O}(5B)}} \frac1{d(w,x_0)}\frac1{\omega(B(x_0,d(w,x_0)))}|b(w)-b_{\mathcal{O}(B)}|^{s'}d\omega(w)\bigg)^{\frac1{s'}}\\
&\qquad\times \bigg(\int_{\mathbb R^N\setminus {\mathcal{O}(5B)}} \frac1{d(w,x_0)}\frac1{\omega(B(x_0,d(w,x_0)))}|f(w)|^sd\omega(w)\bigg)^{\frac1s},
\end{align*}
where $1<s<p<\infty$.
Hence,
\begin{align*}
&\int_{\mathbb R^N\setminus {\mathcal{O}(5B)}} \frac1{d(x_0,w)}\frac1{\omega(B(x_0,d(w,x_0)))}|f(w)|^sd\omega(w) \\
&\lesssim \sum_{j=0}^\infty \int_{2^j5r\le d(w,x_0)\le 2^{j+1}5r}\frac1{ d(w,x_0)}\frac1{\omega(B(x_0,d(w,x_0)))}|f(w)|^sd\omega(w) \\
&\lesssim \sum_{j=0}^\infty 2^{-j}r^{-1}\frac1{\omega(B(x_0,2^{j+1}5r))}\int_{ d(w,x_0)\le 2^{j+1}5r} |f(w)|^sd\omega(w)\\
&\lesssim r^{-1}M_d ( |f|^s )(x).
\end{align*}
Similarly, by the John-Nirenberg inequality, we have
\begin{align*}
&\int_{\mathbb R^N\setminus {\mathcal{O}(5B)}} \frac1{d(x_0,w)}\frac1{\omega(B(x_0,d(w,x_0)))}|b(w)-b_{\mathcal{O}(5B)}|^{s'}d\omega(w). \\
&\lesssim \sum_{j=0}^\infty 2^{-j}r^{-1}\frac1{\omega(B(x_0,2^{j+1}5r))}\int_{d(w,x_0)\le 2^{j+1}5r} |b(w)-b_{\mathcal{O}(5B)}|^{s'}d\omega(w)\\
&\lesssim r^{-1}\|b\|_d^{s'}.
\end{align*}
Thus,
$$|{\rm I\!I\!I}(y)-{\rm I\!I\!I}(x_0)|\lesssim \|b\|_d\Big(M_d ( |f|^s )(x)\Big)^{\frac1s}.$$
Therefore,
\begin{align*}
{1\over \omega(B)}\int_B |{\rm I\!I\!I}(y) - {\rm I\!I\!I}_B |d\omega(y)
&\leq {2\over \omega(B)}\int_B |{\rm I\!I\!I}(y)-{\rm I\!I\!I}(x_0) |d\omega(y) \\
&\lesssim \|b\|_d\Big(M_d ( |f|^s )(x)\Big)^{\frac1s}.
\end{align*}
By the above estimates we obtain that
$$|([b,R_j]f)^\sharp(x)|\lesssim \|b\|_d\bigg(\Big(M ( |R_jf|^s )(x)\Big)^{1\over s}+ \Big(M ( |f|^\beta )(x)\Big)^{1\over \beta} + \Big(M_d ( |f|^s )(x)\Big)^{\frac1s} \bigg).$$
Since $M, M_d$ and $R_j$ are bounded on $L^p(\mathbb R^N, d\omega)$, we obtain
$$\|[b,R_j]f\|_{L^p(\mathbb R^N, d\omega)}\lesssim \|([b,R_j]f)^\sharp\|_{L^p(\mathbb R^N, d\omega)}\lesssim \|b\|_d\|f\|_{L^p(\mathbb R^N, d\omega)}.$$
The upper bound of commutator is complete.
\end{proof}

\bigskip
\maketitle
\section{Proof of Theorem \ref{commutator}: Lower bound of commutator}

In this section, we want to prove the lower bound of commutator $[b,R_j]$.

\begin{definition}\label{mfb-N}
Let $f$ be finite almost everywhere on $\mathbb R^N$. For $B\subseteq \mathbb R^N$ 
with $\omega(B)<\infty$, we define a median value $m_f(B)$ of $f$ over $B$ to be a real number
satisfying
$$\omega\{x\in B : f(x)>m_f(B)\})\le \frac12\omega(B)\qquad\mbox{and}\qquad 
\omega(\{x\in B : f(x)<m_f(B)\})\le \frac12\omega(B).$$
\end{definition}

\begin{proof}[Proof of Theorem \ref{commutator}: lower bound of commutator]
For given $b\in L^1_{\mbox{loc}}(\mathbb R^N, d\omega)$ and for any ball $B$,
let $\Omega_N(b, B)$ be the oscillation defined by
$$\Omega_N(b, B):= \frac1{\omega(B)}\int_B |b(x)-b_B|d\omega(x),$$
where $b_B$ is the average value of $b$ in $B$. Under the assumption of Theorem \ref{commutator}, we will show that for any ball $B$,
\begin{equation}\label{eq 3.12}
|\Omega_N(b, B)|\lesssim 1.
\end{equation}
Let $B=B(x_0,r)$ with $x_0\in \mathbb R^N$ and $r>0$. Note that
\begin{align*}
[b, R_j]f(x)&= b(x)R_jf(x)-R_j(bf)(x)\\
 &=\int_{\mathbb R^N}  (b(x)-b(y))R_j(x,y)f(y)d\omega(y),
\end{align*}

\noindent where $$R_j(x,y)=-c\int_0^\infty \frac{y_j-x_j}t h_t(x,y)\frac{dt}{\sqrt t}.$$
We choose $\widetilde B=B(\tilde x_0, r)$ such that $y_j-x_j\ge r$ and $\|x-y\|\approx r$ for $x\in B$  and $y\in \widetilde B$. 
Then based on Definition \ref{mfb-N}, 
we now choose two measurable sets
 $$E_1\subset\{y\in \widetilde B : b(y)< m_b(\widetilde B)\}\qquad\mbox{and}\qquad 
    E_2\subset\{y\in \widetilde B : b(y)\ge m_b(\widetilde B)\}$$
such that $\omega(E_i) =  \frac12\omega(\widetilde B), i=1,2$,  and that 
$E_1\cup E_2=\tilde B$, $E_1 \cap E_2 =\emptyset$.

Moreover, we define
$$B_1:=\{x\in B : b(x)\ge m_b(\widetilde B)\}\qquad\mbox{and}\qquad 
    B_2:=\{x\in B : b(x)\le m_b(\widetilde B)\}.$$
Now based on the definition of $E_i$ and $B_i$, for $(x,y)\in B_i\times E_i, i=1,2,$ we have
\begin{align*}
|b(x)-b(y)|&=|b(x)-m_b(\tilde B)+m_b(\widetilde B)-b(y)|\\
   &=|b(x)-m_b(\widetilde B)|+|m_b(\widetilde B)-b(y)|\ge |b(x)-m_b(\widetilde B)|.
\end{align*}
Hence, we have the following facts.
\begin{equation}\label{eq 3.13}
\begin{aligned}
& \mbox{(i) } B=B_1\cup B_2, \widetilde B= E_1\cup E_2 
    \mbox{ and }\omega(E_i)\ge \frac12\omega(\widetilde B), i=1,2;\\
& \mbox{(ii) } b(x)-b(y) \mbox{ does not change sign for all }(x,y)\in B_i\times E_i, i=1,2; \\
&\mbox{(iii) } |b(x)-m_b(\tilde B)|\le |b(x)-b(y)| \mbox{ for all }(x,y)\in B_i\times E_i, i=1,2.
\end{aligned}
\end{equation}
By Theorem \ref{th pointwise lower}, we obtain that, for $(x,y)\in B_i\times E_i, i=1,2$, 
$$|R_j(x,y)|\ge \frac 1{\omega(B(x_
     0,r))}.$$
Let $f_i=\chi_{E_i}, i=1,2.$ 
Then the facts \eqref{eq 3.13} give
\begin{align*}
\frac1{\omega(B)}\sum_{i=1}^2\int_B |[b, R_j]f_i(x)|d\omega(x)
 &\ge \frac1{\omega(B)}\sum_{i=1}^2\int_{B_i} |[b, R_j]f_i(x)|d\omega(x) \\
 &= \frac1{\omega(B)}\sum_{i=1}^2\int_{B_i}\int_{E_i} |b(x)-b(y)||R_j(x,y)| d\omega(y)d\omega(x)\\
 &\gtrsim  \frac1{\omega(B)}\sum_{i=1}^2\int_{B_i}|b(x)-m_b(\widetilde B)|\frac 1{\omega(B(x_
     0,r))}\int_{E_i} d\omega(y)d\omega(x)\\
  &\gtrsim  \frac1{\omega(B)}\sum_{i=1}^2\int_{B_i}|b(x)-m_b(\widetilde B)|d\omega(x)\\
 &\gtrsim |\Omega_N(b, B)|.
\end{align*}
On the other hand, from H\"older's inequality and the boundedness of $[b, R_j]$,
we deduce that
\begin{align*}
\frac1{\omega(B)}\sum_{i=1}^2\int_B |[b, R_j]f_i(x)|d\omega(x)
&\lesssim \frac1{\omega(B)}\sum_{i=1}^2 \bigg(\int_B |[b, R_j]f_i(x)|^pd\omega(x)\bigg)^{1/p}\omega(B)^{1/p'} \\
&\lesssim  \frac1{\omega(B)}\sum_{i=1}^2 \|[b,R_j]\|_{L^p(\mathbb{R}^N, \omega) \to L^p(\mathbb{R}^N,\omega)}\omega(E_i)^{1/p} \omega(B)^{1/p'} .
\end{align*}
Since  $\|x-y\|\approx r$ for $x\in B$  and $y\in \widetilde B$, we have
$\omega(\widetilde B)\lesssim \omega(B)$ and then
$$\frac1{\omega(B)}\sum_{i=1}^2\int_B |[b, R_j]f_i(x)|d\omega(x)
\lesssim  \|[b,R_j]\|_{L^p(\mathbb{R}^N, \omega) \to L^p(\mathbb{R}^N, \omega)}.$$
Therefore,
$$|\Omega_N(b, B)|\lesssim \|[b,R_j]\|_{L^p(\mathbb{R}^N, \omega) \to L^p(\mathbb{R}^N, \omega)}.$$
The proof is complete.
\end{proof}

\section{Proof of Theorem \ref{compact} : The compactness of $[b,R_j]$}

It follows from \cite{CDLW} that the ${\rm VMO}_d(\mathbb R^N)$ are equivalent the 
 ${\rm BMO}_d$-closure of the set $\Lambda_{d,0}(\mathbb R^N)$ of $\Lambda_d (\mathbb R^N)$, the Lipschitz function space on space of homogeneous type $(\mathbb R^N, d, d\omega)$, with the compact support. 

\bigskip
\noindent
{\bf Sufficiency}:

A set $S$ is \emph{precompact} if its closure is compact. A common way to check precompactness is to use the Riesz--Kolmogorov theorem \cite[Theorem 1]{GM}, which we recall in below.
\begin{thm}[\cite{GM}]\label{t-fre kol}(\textbf{Riesz--Kolmogorov theorem})
Let $\mu$ be a doubling measure such that 
$$h(r):=\inf\{\mu(B(x,r)): x\in X\}>0\qquad\text{for each $r>0$}$$
and assume $1<p<\infty$. Let $x_0\in X$, then the subset $E$ of $L^p(X, \mu)$ is relatively compact if and only if
the following conditions are satisfied:

{\rm(a)}\ $E$ is bounded;

{\rm(b)}\ $$\lim_{R\to \infty}\int_{X\setminus B(x_0,R)}|f(x)|^pd\mu(x)=0
\qquad\mbox{uniformly for $f\in E$},$$

{\rm(c)}\  $$\lim_{r\to 0}\int_X |f(x)-f_{B(x,r)}|^pd\mu(x)=0\qquad\mbox{uniformly for $f\in E$}.$$
\end{thm}

Let $X=\mathbb R^N$ and $\mu=\omega$.
Then $(\mathbb R^N, \|\cdot\|, \omega)$ is metric space with doubling measure. 
%Let $X=\mathbb R^N/ G$ be the space of orbits equipped with the metric $d(\mathcal{O}(x),\mathcal{O}(y))=d(x,y)$
%and the measure ${\bf m}(A)=\omega\big(\cup_{\mathcal{O}(x)\in A} \mathcal{O}(x)\big)$. So $(X, d, {\bf m})$ is the space
%of homogeneous type of in the sense of Coifman--Weiss. See \cite[page 2403]{ADH}.
Note that 
%$$\omega(B(x,r))\leq \omega\big(\mathcal{O}(B(x,r))\big)\leq |G|\omega(B(x,r))$$
%and
$$\omega(B(x,r))\sim r^N \prod_{\alpha\in R}\big(|\langle\alpha,x\rangle|+r\big)^{\kappa(\alpha)}\ge r^{\bf N}.$$
Thus, we see that  
$\inf\{\omega(B(x,r)): x\in \mathbb R^N\}>0$ for each $r>0$.

We first show that  when $b\in {\rm VMO}_d(\mathbb R^N)$, the commutator $[b, R_j]$ is compact on $L^p(\mathbb R^N)$.
By a density argument, it suffices to show that $[b, R_j]$ is a compact operator for $b\in \Lambda_{d,0}(\mathbb R^N)$.

For $b\in \Lambda_{d,0}(\mathbb R^N)$, to show $[b,R_j]$ is compact on $L^p(\mathbb R^N)$, it suffices to show that for every bounded subset $E\subset L^p(\mathbb R^N)$, the set
$[b,R_j]E$ is precompact. Thus, we only need to show that $[b,R_j)]E$ satisfies the
hypotheses (a)--(c) in the  Riesz--Kolmogorov Theorem (Theorem \ref{t-fre kol}). We first point out that by Theorem \ref{commutator}
and the fact that $b\in {\rm BMO}_d(\mathbb R^N)$, $[b,R_j]$ is bounded on $L^p(\mathbb R^N)$, which implies that $[b,R_j]E$
satisfies hypothesis (a) in Theorem \ref{t-fre kol}.

Next, we will show that~$[b,R_j]E$ satisfies hypothesis (b) in Theorem \ref{t-fre kol}.
We may assume that~$b\in \Lambda_{d,0}(\mathbb R^N)$ with~$\supp b \subset {\mathcal O}(B(0,R))$.
For~$t>2$, set~$K^c := \{x \in \mathbb R^N: d(x,0) > tR\}.$
There exists an increasing function $\phi$ such that $ \{x \in \mathbb R^N: d(x,0) \le tR\}\subseteq
 B(0, \phi(tR))$.  
Then we have
\begin{align*}
  \|[b,R_j]f\|_{L^p(B(0, \phi(tR))^c, d\omega)}
  &\le \|[b,R_j]f\|_{L^p(K^c, d\omega)} 
     =\|bR_j(f) -  R_j(bf)\|_{L^p(K^c, d\omega)}\\
   &\leq  \|bR_j(f)\|_{L^p(K^c, d\omega)} +  \|R_j(bf)\|_{L^p(K^c, d\omega)}.
\end{align*}
Since~$\supp b\cap K^c = \emptyset$,  we have
\[ \int_{d(x,0) > tR}|b R_j(f)(x)|^p \,d \omega(x)  = 0,\]
and so
\begin{equation}\label{eq:thm2.40}
\|[b,R_j]f\|_{L^p(K^c, d\omega)} \leq \|R_j(bf)\|_{L^p(K^c, d\omega)}.
\end{equation}
Using the size condition of $R_j(x,y)$ and the fact that~$\supp b \subset {\mathcal O}(B(0,R))$ we have
\begin{align}\label{eq:thm2.41}
  |R_j(bf)(x)| &\leq  \int_{d(y,0)<R} |R_j(x,y)|| b(y)||f(y)| \,d \omega(y)\nonumber\\
   &\leq   \int_{d(y,0)<R} \frac{1}{\omega(B(x,d(x,y)))}| b(y)||f(y)| d \omega(y).
\end{align}
Notice that for~$d(x,0)>tR$, $t >2$ and~$d(y,0)<R$ 
we have~$d(x,y)>d(x,0)-d(y,0)>d(x,0)/2$.
Using this and H\"{o}lder's inequality, inequality~\eqref{eq:thm2.41} yields
\begin{align*}
  |R_j(bf)(x)|
  &\leq C\frac{1}{\omega(B(x,d(x,0)))} \int_{d(y,0)<R} | b(y)||f(y)| \,d \omega(y) \\
   &\leq  C\frac{1}{\omega(B(x,d(x,0)))}
   \Big(\int_{d(y,0)<R} | b(y)|^{p'} \,d \omega(y)\Big)^{1/p'} \Big(\int_{d(y,0)<R} | f(y)|^{p} \,d \omega(y)\Big)^{1/p}  \\
   &\leq C\frac{1}{\omega(B(0,d(x,0)))} \|b\|_{L^{\infty}(\mathbb R^N)}\|f\|_{L^{p}(\mathbb R^N, d\omega))} \omega({\mathcal O}(B(0,R)))^{1/p'} , 
 \end{align*}
since~$b \in \Lambda_{d,0}(\mathbb R^N)$ and $\omega(B(x,d(x,0)))\approx \omega(B(0,d(x,0)))$. Using this estimate of~$|R_j(bf)(x)|$, \eqref{eq:thm2.40} becomes
\begin{align*}
 &\|[b,R_j]f(x)\|_{L^p(K^c, d\omega)} \\
  &\leq  C\omega(B(0,R))^{1/p'} \|b\|_{L^{\infty}(\mathbb R^N)}\|f\|_{L^{p}(\mathbb R^N, d\omega))}  \Big( \int_{d(x,0)>tR} \frac{1}{{\omega(B(0,d(x,0)))^p}} \,d\omega(x)\Big)^{1/p}\\
 &\leq  C\omega(B(0,R))^{1/p'} \|b\|_{L^{\infty}(\mathbb R^N)}\|f\|_{L^{p}(\mathbb R^N, d\omega))}  \Big(\sum_{j=0} ^\infty\int_{2^{j}tR<d(x,0)<2^{j+1}tR} \frac{1}{{\omega(B(0,d(x,0)))^p}} \,d\omega(x)\Big)^{1/p}\\
 &\le C\omega(B(0,R))^{1/p'} \|b\|_{L^{\infty}(\mathbb R^N)}\|f\|_{L^{p}(\mathbb R^N, d\omega))} \sum_{j=0}^\infty \frac{\omega(B(0,2^{j+1}tR))^{1/p}}{\omega(B(0,2^jtR))}\\
 &\le C\|b\|_{L^{\infty}(\mathbb R^N)}\|f\|_{L^{p}(\mathbb R^N, d\omega))}
 2^{N/p}t^{-N/p'}\sum_{j=1}^\infty 2^{-Nj/p'},
\end{align*}
where the last inequality follows from \eqref{w-double 2}.
Finally, given each~$\e > 0$, we can choose~$t$ large enough such that~$C2^{N/p}t^{-N/p'\sum_{j=1}^\infty 2^{-Nj/p'}}< \e$. Here the constant~$C$ depends on~$b$ and on the bound on~$\|f\|_{L^{p}(\mathbb R^N, d\omega))}$ for~$f \in E$.
Hence hypothesis (b)  in Theorem \ref{t-fre kol} holds for $[b,R_j]E$.

It remains to prove that $[b,R_j]E$ also satisfies hypothesis (c) of Theorem \ref{t-fre kol}.
Let $\e$ be a fixed positive constant in $(0,1/8)$.
Since~$b \in \Lambda_{d,0}({\mathbb R}^N)$, $b$ is uniformly continuous.  Choose~$r = r(b,\e)$ sufficiently small that for all~$z \in B(x,r)$, we have both~$\|x-z\|<\e^2$ and for all~$x\in {\mathbb R}^N$, $|b(x) - b(z)| < \e$. Fix~$z \in  B(x,r)$.
Then for all~$x\in {\mathbb R}^N$,
$$
[b, R_j]f(x)-([b, R_j]f)_{B(x,r)}
=\frac1{\omega(B(x,r))}\int_{B(x,r)} \big([b, R_j]f(x)-[b, R_j]f(z)\big)d\omega(z).
$$
Note that 
\begin{align*}
&[b, R_j]f(x)-[b, R_j]f(z)\\
&\quad=\int_{{\mathbb R}^N} R_j(x, y)[b(x)-b(y)] f(y)\,d\omega(y)
-\int_{{\mathbb R}^N} R_j(z, y)[b(z)-b(y)] f(y)\,d\omega(y)\\
&\quad=\int_{d(x,y)>\e^{-1} \|x-z\|}R_j(x, y)[b(x)-b(z)]f(y)\,d\omega(y)\\
&\quad\quad+\int_{d(x,y)>\e^{-1} \|x-z\|}[R_j(x, y)-R_j(z,y)][b(z)-b(y)]f(y)\,d\omega(y)\\
&\quad\quad+\int_{d(x,y)\le \e^{-1} \|x-z\|}R_j(x,y)[b(x)-b(y)]f(y)\,d\omega(y)\\
&\quad\quad-\int_{d(x,y)\le \e^{-1} \|x-z\|}R_j(z,y)[b(z)-b(y)]f(y)\,d\omega(y)\\
&\quad=:\sum_{i=1}^4{\rm L}_i(x,z).
\end{align*}
We begin with estimating ${\rm L}_2$. Since $\e\in(0, 1/2)$, it follows that
\[d(x,y)>\e^{-1} \|x-z\| \Rightarrow \|x-z\|  < \frac{d(x,y)}{2}.\]
Thus we may apply the smoothness condition of the kernel~$R_j(x,y)$ (as in Theorem \ref{smooth r}), concluding that
 \[ |R_j(x,y) - R_j(z,y)| \leq \frac{\|x-z\|}{\|y-x \|}\frac{1}{\omega(B(x,d(x,y)))} \le \frac{\|x-z\|}{d(x,y)}\frac{1}{\omega(B(x,d(x,y)))}.\]
Using this inequality, together with the fact that~$b\in \Lambda_{d,0}({\mathbb R}^N)$,  we have
\begin{align*}
  |{\rm L}_2(x,z)| 
 & \le \|x-z\| \int_{d(x,y)>\e^{-1} \|x-z\|} \frac{|f(y)|}{d(x,y)\omega(B(x,d(x,y)))} \,d\omega(y)\\
  &\le \|x-z\| \sum_{j=0}^\infty  \int_{2^j\e^{-1} \|x-z\|<d(x,y)<2^{j+1}\e^{-1} \|x-z\|}
  \frac{|f(y)|}{d(x,y)\omega(B(x,d(x,y)))} \,d\omega(y)\\
 &\le \e \sum_{j=0}^\infty 2^{-j}\frac1{\omega(B(x,2^j\e^{-1} \|x-z\|))}
  \int_{d(x,y)<2^{j+1}\e^{-1} \|x-z\|} |f(y)|   \,d\omega(y)\\
 &\le C\e \sum_{j=0}^\infty 2^{-j}M_\omega(f)(x)\\
 &\le C\e M_\omega(f)(x),
\end{align*}
where $M_\omega$ is the Hardy--Littlewood maximal operator on $(\mathbb R^N, d, \omega)$.
Hence, 
$$\frac1{\omega(B(x,r))}\int_{B(x,r)} |{\rm L}_2(x,z)|d\omega(z)\le C\e M_\omega(f)(x).$$
This further gives
\begin{equation}\label{eq:thm2.36}
\begin{aligned}
&\int_{{\mathbb R}^N}\bigg|\frac1{\omega(B(x,r))}\int_{B(x,r)} |{\rm L}_2(x,z)|
               d\omega(z)\bigg|^p\,d\omega(x) \\
&\le C\e^p \int_{{\mathbb R}^N}|M_\omega(f)(x)|^p\,d\omega(x) \\
&\le C\e^p\|f\|^p_{L^p({\mathbb R}^N, d\omega)}.
\end{aligned}
\end{equation}
Turning to ${\rm L}_3$, by the size condition of the kernel~$R_j(x,y)$ (as in Theorem \ref{smooth r}) and the fact that $b\in \Lambda_{d,0}({\mathbb R}^N)$, we conclude that
\begin{align*}
  |{\rm L}_3(x,z)| 
   &\lesssim \int_{d(x,y)\le \e^{-1} \|x-z\|}  \frac{|f(y)|d(x,y)}{\omega(B(x,d(x,y)))}\,d\omega(y) \\
   &=\sum_{i=-\infty}^{-1} \int_{2^{i}\e^{-1} \|x-z\|<d(x,y)\le 2^{i+1}\e^{-1} \|x-z\|}  \frac{|f(y)|d(x,y)}{\omega(B(x,d(x,y)))}\,d\omega(y) \\
   &\lesssim \sum_{i=-\infty}^{-1} 2^{i+1}\e^{-1} \|x-z\|\frac1{\omega(B(x,2^{i}\e^{-1} \|x-z\|))}\int_{d(x,y)\le 2^{i+1}\e^{-1} \|x-z\|}  |f(y)|\,d\omega(y) \\
   &\lesssim \e^{-1} \|x-z\|M_\omega(f)(x)
\end{align*}
By our choice of~$z$, 
\begin{equation}\label{eq:thm2.37}
\int_{{\mathbb R}^N}\bigg|\frac1{\omega(B(x,r))}\int_{B(x,r)} |{\rm L}_3(x,z)|
               d\omega(z)\bigg|^p\,d\omega(x) 
\lesssim \e^p\|f\|^p_{L^p({\mathbb R}^N, d\omega)}.
\end{equation}
Note that $\|z-\sigma(y)\|\le \|x-\sigma(y)\|+\|x-z\|$ for $
\sigma\in G$ gives $d(z, y)\le d(x,y)+\|x-z\|$. We have $d(z,y)\le 2\e^{-1} \|x-z\|$ if $d(x,y)<\e^{-1} \|x-z\|$ and $0<\e<1/8$. 
Hence,
\begin{align*}
 |{\rm L}_4(x,z)| &\lesssim \int_{d(z,y)\le 2\e^{-1} \|x-z\|} |R_j(z,y)||b(z)-b(y)||f(y)|\,d\omega(y)\\
 &\lesssim \int_{d(x,y)\le 2\e^{-1} \|x-z\|}  \frac{|f(y)|d(z,y)}{\omega(B(z,d(z,y)))}\,d\omega(y) \\
   &=\sum_{i=-\infty}^{-1} \int_{2^{i}\e^{-1} \|x-z\|<d(z,y)\le 2^{i+1}\e^{-1} \|x-z\|}  \frac{|f(y)|d(z,y)}{\omega(B(z,d(z,y)))}\,d\omega(y) \\
   &\lesssim \sum_{i=-\infty}^{0} 2^{i+1}\e^{-1} \|x-z\|\frac1{\omega(B(z,2^{i}\e^{-1} \|x-z\|))}\int_{d(x,y)\le 2^{i+1}\e^{-1} \|x-z\|}  |f(y)|\,d\omega(y) \\
   &\lesssim \e^{-1} \|x-z\|M_\omega(f)(z)
\end{align*}
and then
$$\frac1{\omega(B(x,r))}\int_{B(x,r)} |{\rm L}_4(x,z)|d\omega(z)
\le\frac\e{\omega(B(x,r))}\int_{B(x,r)} M_\omega(f)(z)d\omega(z) \le C\e M_{non}(M_\omega(f))(x),$$
where $M_{non}$ is the non-central Hardy--Littlewood maximal operator on $(\mathbb R^N, \|\cdot\|, \omega)$.
This implies that
\begin{equation}\label{eq:thm2.38}
\int_{{\mathbb R}^N}\bigg|\frac1{\omega(B(x,r))}\int_{B(x,r)} |{\rm L}_4(x,z)|
               d\omega(z)\bigg|^p\,d\omega(x) 
\lesssim \e^p\|f\|^p_{L^p({\mathbb R}^N, d\omega)}.
\end{equation}

As the last step, we consider $\rm{L}_1$:
$$
  |{\rm L_1}(x,z)| 
   \leq |b(x)-b(z)|\sup_{t>0}\bigg|\int_{d(x,y)>t}R_j(x, y)f(y)\,d\omega(y) \bigg|.$$
Thanks to \cite[Theorem 1.3]{THL}, we choose $0<r<p$ such that 
$$\sup_{t>0}\bigg|\int_{d(x,y)>t}R_j(x, y)f(y)\,d\omega(y) \bigg|
\lesssim M(|R_j(f)|^r)(x)^{1/r}+\sum_{\sigma\in G}M(f(\sigma(x))).$$
Recall that~$|b(x) - b(z)| < \e$ by our choice of~$z$.
Hence
\begin{equation}\label{eq:thm2.39}
\begin{aligned}
 &\int_{{\mathbb R}^N}\bigg|\frac1{\omega(B(x,r))}\int_{B(x,r)} |{\rm L}_1(x,z)|
               d\omega(z)\bigg|^p\,d\omega(x) \\
  &\lesssim C_p\sum_{\sigma\in G}\int_{{\mathbb R}^N}  \e^p\Big( M(|R_j(f)|^r)(x)^{p/r} +M(f(\sigma(x)))^p\Big) \,d\omega(x) \\
   &\lesssim C_p\e^p \int_{{\mathbb R}^N}    \Big(  M(|R_j(f)|^r)(x)^{p/r} +M(f(\sigma(x)))^p\Big)\,d\omega(x) \\
  &\lesssim C_p\e^p\|f\|^p_{L^p({\mathbb R}^N, d\omega)}.
\end{aligned}
\end{equation}
Combining the estimates~\eqref{eq:thm2.36}--\eqref{eq:thm2.39} of ${\rm L}_i,\,i\in\{1, 2,3,4\}$, we conclude that
\begin{align*}
\int_{{\mathbb R}^N}\left|[b, R_j]f(x)-[b, R_j]f(z)\right|^p\,d\omega(x)
&\lesssim C_p\e^p\|f\|^p_{L^p({\mathbb R}^N, d\omega)}.
\end{align*}

\noindent This shows that $[b,R_j]E$ satisfies hypothesis (c) in Theorem \ref{t-fre kol}. Hence,
$[b, R_j]$ is a compact operator.

\bigskip
\noindent
{\bf Necessity}:\\

We start by assuming $b\in {\rm BMO}_{Dunkl}(\mathbb R^N)$ is such that $[b,R_k]$ is compact from $L^{p}(\mathbb R^N, d\omega)$ to $L^{p}(\mathbb R^N, d\omega)$. We will use the method of proof by contradiction and  hence let us suppose that $b\notin {\rm VMO}_{Dunkl}(\mathbb R^N)$. Here we follow the main idea from \cite{LL}.

As we assume that  $b\notin {\rm VMO}_{Dunkl}(\mathbb R^N)$, at least one of the three conditions presented in the definition of ${\rm VMO}_{Dunkl}(\mathbb R^N)$ fails to hold. Since a similar argument will work for all three conditions, let us suppose that the first condition does not hold.

That is, there exists some $\delta_0>0$ and a sequence of balls $\{Q_i\}_{i\in I} \subset \mathbb R^N$ such that $l(Q_i) \rightarrow 0$ as $i \rightarrow \infty$ and we have that
\begin{equation}\label{seqvmo1}
{1\over \omega(Q)}\int_{Q}\left|b(x)-b_{Q}\right|d\omega(x) \geq \delta_0.
\end{equation}
We will also further assume without loss of generality that
\begin{equation}\label{lengthcon1}
4l(Q_{j_{i+1}})\leq l(Q_{j_{i}}).
\end{equation}
Note that 
$$R_k(x,y)=-c\int_0^\infty \frac{y_k-x_k}t h_t(x,y)\frac{dt}{\sqrt t}$$
and $ h_t(x,y)$ has lower bounded $\frac 1{\min\{\omega(B(x,\sqrt t)), \omega(B(y,\sqrt t))\}}e^{-c\|x-y\|^2/t}$.
We choose $\widetilde Q_j=Q_j(\tilde x_0, r)$ such that $y_k-x_k\ge r$ and $\|x-y\|\approx r$ for $x\in Q_j$  and $y\in \widetilde Q_j$.
Let us denote by $m_{b}(\tilde{Q}_j)$ a median value of $b$ on the ball $\tilde{Q}_j$. That is $m_{b}(\tilde{Q}_j)$ is a real number such that the two sets below have a measure at least $\frac{1}{2}\omega(\tilde{Q}_j)$
\begin{equation}
    F_{j,1} \subset \{y\in \tilde{Q}_j: b(y) \leq m_{b}(\tilde{Q}_j)\},\quad  F_{j,2} \subset \{y\in \tilde{Q}_j: b(y) \geq m_{b}(\tilde{Q}_j)\}.
\end{equation}

Also define the sets
\begin{equation}
    E_{j,1} \subset \{x\in {Q}_j: b(x) \geq m_{b}(\tilde{Q}_j)\},\quad  E_{j,2} \subset \{x\in {Q}_j: b(x) < m_{b}(\tilde{Q}_j)\}.
\end{equation}
So we have that $Q_j = E_{j,1}\cup E_{j,2}$ and $E_{j,1}\cap E_{j,2} = \emptyset$ and we also have the following
\begin{align*}
    b(x) - b(y) \geq 0,& \quad (x,y) \in E_{j,1}\times F_{j,1}\\
    b(x) - b(y) < 0,& \quad (x,y) \in E_{j,2}\times F_{j,2}.
\end{align*}
For $(x,y) \in  E_{j,1}\times F_{j,1}\cup  E_{j,1}\times F_{j,1}$, we have that
\begin{equation*}
    |b(x)-b(y)| = |b(x)-m_{b}(\tilde{Q}_j)| + |m_{b}(\tilde{Q}_j) - b(y)| \geq |b(x)-m_{b}(\tilde{Q}_j)|.
\end{equation*}

Define the following sets
\begin{equation}
 \tilde F_{j,1} := F_{j,1}\setminus \cup_{l= j+1}^{\infty}\tilde{Q_l} \quad \tilde F_{j,2}  := F_{j,2}\setminus \cup_{l= j+1}^{\infty}\tilde{Q_l} \quad \forall j = 1,2,\ldots .
\end{equation}
Now using the decay condition for the lengths of $\{Q_j\}$ as given by \eqref{lengthcon1}, we have for each $j$ the following
\begin{equation}\label{setmeasure}
\omega(\tilde F_{j,1}) \geq \omega(F_{j,1}) - \omega(\cup_{l= j+1}^{\infty}\tilde{Q_l})\geq \frac{1}{2} \omega(\tilde{Q_j})- \sum_{l= j+1}^{\infty}\omega(\tilde{Q_l}) \geq \frac{1}{2} \omega(\tilde{Q_j}) - \frac{1}{3} \omega(\tilde{Q_j}) = \frac{1}{6} \omega(\tilde{Q_j}).
\end{equation}

We can obtain a similar estimate for the set $\tilde F_{j,2}$. Observe now for every $j$, the following holds
\begin{align}
 &{1\over \omega(Q_j)}\int_{Q_j}\left|b(x)-b_{Q}\right|d\omega(x) \leq {2\over \omega(Q_j)}\int_{Q_j}\left|b(x)-m_{b}(\tilde{Q}_j)\right|d\omega(x) \\
 &= {2\over \omega(Q_j)}\int_{E_{j,1}}\left|b(x)-m_{b}(\tilde{Q}_j)\right|d\omega(x) + {2\over \omega(Q_j)}\int_{E_{j,2}}\left|b(x)-m_{b}(\tilde{Q}_j)\right|d\omega(x)\nonumber.
\end{align}

From \eqref{seqvmo1} we have that at least one of these inequalities holding
\begin{equation*}
 {2\over \omega(Q_j)}\int_{E_{j,1}}\left|b(x)-m_{b}(\tilde{Q}_j)\right|d\omega(x) \geq \frac{\delta_0}{2},  \quad {2\over \omega(Q_j)}\int_{E_{j,2}}\left|b(x)-m_{b}(\tilde{Q}_j)\right|d\omega(x) \geq \frac{\delta_0}{2}.
\end{equation*}
Let us suppose that the first of these inequalities holds, i.e.,
\begin{equation*}
 {2\over \omega(Q_j)}\int_{E_{j,1}}\left|b(x)-m_{b}(\tilde{Q}_j)\right|d\omega(x) \geq \frac{\delta_0}{2}.
\end{equation*}
Hence for every $j$, using \eqref{setmeasure} we have that
\begin{align}
    &\frac{\delta_0}{4} \leq {1\over \omega(Q_j)}\int_{E_{j,1}}\left|b(x)-m_{b}(\tilde{Q}_j)\right|d\omega(x) \lesssim {1\over \omega(Q_j)}\frac{\omega(\tilde F_{j,1})}{\omega(Q_j)}\int_{E_{j,1}}\left|b(x)-m_{b}(\tilde{Q}_j)\right|d\omega(x)\\
    &\lesssim {1\over \omega(Q_j)}\int_{E_{j,1}}\int_{\tilde F_{j,1}}\frac{1}{\omega(Q_j)}\left|b(x)-m_{b}(\tilde{Q}_j)\right|d\omega(y)d\omega(x).\nonumber
\end{align}
Hence,
\begin{align}
&\delta_0 \lesssim   {1\over \omega(Q_j)}\int_{E_{j,1}}\left|\int_{\tilde F_{j,1}}\frac{1}{\omega(Q_j)}\Big(b(x)-m_{b}(\tilde{Q}_j)\Big)d\omega(y)\right|d\omega(x)\\
&\lesssim {1\over {\omega(Q_j)^{1\over p}\omega(Q_j)^{1\over p'}}}\int_{E_{j,1}}|[b,R_k](\chi_{\tilde F_{j,1}})(x)|d\omega(x)\nonumber\\
& = {1\over \omega(Q_j)^{1\over p'}}\int_{E_{j,1}}\bigg|[b,R_k]\bigg({\chi_{\tilde F_{j,1}}\over \omega(Q_j)^{1\over p}}\bigg)(x)\bigg|d\omega(x).\nonumber
\end{align}

Consider $f_j =: {\chi_{\tilde F_{j,1}}\over \omega(Q_j)^{1\over p}}$, observe that this is a sequence of disjointly supported functions and using \eqref{setmeasure} also satisfy $\|f_j\|_{L^{p}(\mathbb R^N, d\omega)} \simeq 1$. Using the H\"older's inequality yields

\begin{align}
  \delta_0 &\lesssim {1\over \omega(Q_j)^{1\over p'}}\int_{E_{j,1}}|[b,R_k](f_j)(x)| d\omega(x)\\
 &\lesssim {1\over \omega(Q_j)^{1\over p'}}\omega(E_{j,1})^{1\over p'}\bigg(\int_{\mathbb R^N}|[b,R_k](f_j)(x)|^{p}  d\omega(x)\bigg)^{1 \over p}\nonumber\\
 &\lesssim \bigg(\int_{\mathbb R^N}|[b,R_k](f_j)(x)|^{p} d\omega(x)\bigg)^{1 \over p}.\nonumber
\end{align}

Let us consider $\psi$ in the closure of $\{[b,R_k](f_j)\}_j$, then we have $\|\psi\|_{L^{p}(\mathbb R^N, d\omega)}\gtrsim 1$. Now choose some $j_i$ such that

\begin{equation*}
  \|\psi-[b,R_k](f_{j_i})\|_{L^{p}(\mathbb R^N, d\omega)} \leq 2^{-i}.
\end{equation*}

To complete the proof consider a non-negative numerical sequence $\{c_i\}$ with $\|\{c_i\}\|_{l^{p'}}<\infty$ but  $\|\{c_i\}\|_{l^1}=\infty$. Then consider $\phi = \sum_{i}c_if_{j_i}\in L^{p}(\mathbb R^N, d\omega)$ and

\begin{align}
   & \left\|\sum_i c_i\psi-[b,R_k]\phi\right\|_{L^{p}(\mathbb R^N, d\omega)} \leq \left\|\sum_{i}^{\infty}c_i(\psi-[b,R_k](f_{j_i}))\right\|_{L^{p}(\mathbb R^N, d\omega)}\\
   &\leq \|c_i\|_{l^{p'}} \bigg[\sum_i \|\psi-[b,R_k](f_{j_i})\|_{L^{p}(\mathbb R^N, d\omega)}^{p}\bigg]^{1\over p}\lesssim 1.\nonumber
\end{align}
Hence we conclude that $\sum_ic_i\psi\in L^{p}(\mathbb R^N, d\omega)$, but $\sum_ic_i\psi$ is infinite on set of positive measure which is contradiction that completes our proof.

%%%%%%%%%%%%%%%%%%%%%%%%%%%%%%%
%%%%%%%%%%%%%%%%%%%%%%%%%
%%%%%%%%%%%%%%%%%%%%%%%%%%%%%%%

%\section{Lower bound example}
%
%We will show that there exists some locally integrable function $b$ such that 
%$$ \|b\|_{BMO_d}=+\infty $$
%and that 
%$$ \|[b, H]f\|_{L^{2}(\mathbb R, d\omega)}<\infty. $$
%
%Let $G=\mathbb Z_2$ and $H$ be Dunkl Hilbert transform.
%Consider $b=H\chi_{[-1,1]}$, then $b\in BMO_{Dunkl}(\mathbb R)$ but does not belong to $BMO_d$. 
%
%We want to check $\frac1{\omega(B)}\int_{{\mathcal O}(2B)} |b(y)-b_B|^rd\omega(y)<\infty$.
%\begin{align*}
%|b(y)-b_B|&\le \frac1{\omega(B)}\int_B |b(y)-b(z)|d\omega(z)\\
%                &\le \frac1{\omega(B)}\int_B \Big|\int_{-1}^1 H(y,u)-H(z,u) d\omega(u) \Big|d\omega(z)
%\end{align*}
%Moreover, it is clear that $b\in L^4(\mathbb R, d\omega)$ since $H$ is bounded on $L^4(\mathbb R, d\omega)$. For $f\in L^2(\mathbb R, d\omega)\cap L^4(\mathbb R, d\omega)$ with $\|f\|_{L^2(\mathbb R, d\omega)}\le 1$, we have
%\begin{align*}
%\|[b, H]f\|_{L^{2}(\mathbb R, d\omega)}
%&\le \|bH(f)\|_{L^{2}(\mathbb R, d\omega)}+\|H(bf)\|_{L^{2}(\mathbb R, d\omega)}\\
%&\le \|b\|^2_{L^{4}(\mathbb R, d\omega)}\|H(f)\|^2_{L^{4}(\mathbb R, d\omega)}+C\|bf\|_{L^{2}(\mathbb R, d\omega)}\\
%&\le C\|b\|^2_{L^{4}(\mathbb R, d\omega)}\|f\|^2_{L^{4}(\mathbb R, d\omega)}+C\|b\|^2_{L^{4}(\mathbb R, d\omega)}\|f\|^2_{L^{4}(\mathbb R, d\omega)}\\
%&\le C\|b\|^2_{L^{4}(\mathbb R, d\omega)}\|f\|^2_{L^{4}(\mathbb R, d\omega)}
%\end{align*}

   \bigskip
\noindent   
{\bf Acknowledgement: } 
%The authors are extremely grateful to the referee for carefully reading and checking the paper, finding a large number of errors and suggesting the corresponding corrections, and more importantly,  for the essential suggestions that led us to very significant improvements of some results that were contained in a preliminary version of this manuscript.

Ji Li would like to thank Jorge Betancor for helpful discussions.

\bigskip
\bigskip

\medskip
\vskip 0.5cm

\noindent Department of Mathematics, Auburn University, AL
36849-5310, USA.

\noindent {\it E-mail address}: \texttt{hanyong@auburn.edu}

\medskip
\vskip 0.5cm

\noindent Department of Mathematics, National Central University Chung-Li 320, Taiwan Republic of China

\noindent {\it E-mail address}: \texttt{mylee@math.ncu.edu.tw}

\medskip
\vskip 0.5cm
\noindent School of Mathematical and Physical Sciences, Macquarie University, NSW, 2109, Australia.

\noindent {\it E-mail address}: \texttt{ji.li@mq.edu.au}

\medskip
\vskip 0.5cm

\noindent Department of Mathematics
Washington University - St. Louis, St. Louis, MO 63130-4899 USA

\noindent {\it E-mail address}: \texttt{bwick@wustl.edu}

\end{document}